 \documentclass{amsart}
\usepackage{amsmath,amsthm,amsfonts, amssymb,amscd}
\usepackage[all]{xy}

\newtheorem{Theorem}{Theorem}[section]
\newtheorem{Proposition}[Theorem]{Proposition}
\newtheorem{Lemma}[Theorem]{Lemma}

\newtheorem{Example}[Theorem]{Example}
\theoremstyle{remark}
\newtheorem{Remark}[Theorem]{Remark}

\numberwithin{Claim}{Theorem}

 \theoremstyle{definition}

\newcommand{\Ker}{\mbox{\rm Ker}}

\newcommand{\id}{\mbox{\rm id}}
\newcommand{\Aff}{\mbox{\rm Aff}}

\long\def\alert#1{\smallskip{\hskip\parindent\vrule%
\vbox{\advance\hsize-2\parindent\hrule\smallskip\parindent.4\parindent%
\narrower\noindent#1\smallskip\hrule}\vrule\hfill}\smallskip}

\begin{document}
\title[Loomis--Sikorski Theorem
and Stone Duality for Effect Algebras]{Loomis--Sikorski Theorem and
Stone Duality for Effect Algebras with Internal State}
\author{David Buhagiar$^1$, Emmanuel Chetcuti$^1$, and Anatolij Dvure\v censkij$^2$}
\date{}
\maketitle
\begin{center}  \footnote{Keywords:
Effect algebra, state, state-operator, Riesz decomposition property,
unital po-group, simplex, Bauer simplex, Choquet simplex,
Loomis--Sikorski Theorem, Stone duality,

AMS classification:  46C15, 81P10, 03G12

The paper has been supported by the Center of Excellence SAS
-~Quantum Technologies, ERDF OP R\&D Projects CE QUTE ITMS
26240120009 and meta-QUTE ITMS 26240120022, the grant VEGA No.
2/0032/09 SAV, by  the Slovak Research and Development Agency under
the contract APVV-0071-06, Bratislava. Third author thanks
University of Malta for hospitality during May 2010. }
Department of Mathematics, Faculty of Sciences,\\
University of Malta, Msida MSD 2080, Malta,\\
 $^2$Mathematical Institute,  Slovak Academy of Sciences,\\
\v Stef\'anikova 49, SK-814 73 Bratislava, Slovakia\\E-mail: {\tt
david.buhagiar@um.edu.mt}, \ {\tt emanuel.chetcuti@um.edu.mt}\\ {\tt
dvurecen@mat.savba.sk}
\end{center}


\begin{abstract}
Recently Flaminio and Montagna, \cite{FlMo}, extended the language
of MV-algebras  by adding a unary operation, called a
state-operator. This notion is introduced here also for effect
algebras. Having it, we generalize the Loomis--Sikorski Theorem  for
monotone $\sigma$-complete effect algebras with internal state. In
addition, we show that the category of   divisible state-morphism
effect algebras  satisfying (RDP) and countable interpolation with
an order determining system  of states is dual to the category of
Bauer simplices $\Omega$ such that $\partial_e \Omega$ is an F-space
\end{abstract}
\date{}
\maketitle

\section{Introduction}

The famous Loomis--Sikorski Theorem  was proved independently by two
authors, Sikorski and Loomis,   \cite{Sik, Loo}, after the Second
World War, and nowadays it has many serious applications in
different areas of mathematics.   Roughly speaking it states that
every $\sigma$-complete Boolean algebra is a $\sigma$-algebra of
subsets of a set up to some modulo, or, precisely every
$\sigma$-complete Boolean algebra is a $\sigma$-epimorphic image of
some $\sigma$-algebra of subsets. It can be rewritten also in the
form that our $\sigma$-algebra is practically an appropriate system
of $[0,1]$-valued functions; in our case it is a system of
characteristic functions, where the Boolean operations on the set of
functions are defined by points.

This was extended also for $\sigma$-complete MV-algebras in
\cite{Dvu1, Mun1, BaWe} showing that every $\sigma$-complete
MV-algebra is a $\sigma$-epimorphic image of a system of
$[0,1]$-valued functions, called a {\it tribe}, where again
MV-operations on functions in the tribe are defined by points.

In the Nineties, the theory of quantum structures was enriched  by
new structures, called {\it effect algebras}, see \cite{FoBe}. They
are inspired by the mathematical foundations of quantum mechanics
(for an overview on effect algebras, see \cite{DvPu}) as well as by
many valued features of quantum mechanical measurements. One of the
most important examples of effect algebras studied in quantum
mechanics is the system ${\mathcal E}(H)$ of all Hermitian operators
$A$ on a Hilbert space $H$ such that $O \le A \le I$, where $O$ and
$I$ are the zero and identity operator on $H$. The category of
effect algebras contains Boolean algebras, orthomodular lattices,
orthomodular posets and orthoalgebras.

The Loomis--Sikorski Theorem was extended also for monotone
$\sigma$-complete effect algebras by the present authors in
\cite{BCD}.

The notion of a state, an analogue of probability measure, is a
basic notion for quantum structures. It is motivated by the notion
of a state in quantum mechanics. The set of all states on an effect
algebra can be a good source of information on the given system, and
it will be also   deeply used in our paper.

Recently, the notion of a state was generalized by \cite{FlMo} to an
algebraically defined notion for MV-algebras.  They enlarged the
language of MV-algebras introducing  a unary operation, $\tau,$
called an {\it internal state} or a {\it state-operator}. Such
MV-algebras are called {\it state MV-algebras}. These algebras are
now intensively studied e.g. in \cite{DiDv, DDL1, DDL2, DDL3}.

One of important properties of a state-operator $\tau$ on an
MV-algebra is $\tau^2 = \tau,$ idempotency. Inspired by this, in the
present paper, we introduce  notions of a (i) state-operator for
effect algebras as an endomorphism $\tau: E \to E$ such that
$\tau\circ \tau = \tau,$ (ii) a strong state-operator, and (iii) a
state-morphism-operator with some additional properties coming from
state MV-algebras. The later two coincide with ones for MV-algebras.
We add it to the language of effect algebras as an internal state
and they will form a so-called {\it state effect algebras}.
Moreover, given an integer $n,$ we introduce an $n$-state-operator
as an endomorphism $\tau: E\to E$ that in $n$-potent, i.e., $\tau^n
= \tau,$ and the couple $(E,\tau)$ is said to be $n$-{\it state
effect algebra}.

Besides the presentation of basic properties of state effect
algebras, we present  the following two main results:

\begin{enumerate}

\item
We generalize the Loomis--Sikorski Theorem for monotone
$\sigma$-complete $n$-state effect algebras with the Riesz
Decomposition Property ((RDP) for short) showing that it is always a
$\sigma$-monotone epimorphic image of  an effect tribe with (RDP),
an appropriate system of $[0,1]$-valued functions that is a
$\sigma$-complete effect algebra with pointwise defined effect
algebraic operations, and  with an $n$-state-operator induced by a
function.

\item
We show that the category   of  divisible state-morphism effect
algebras  satisfying (RDP) and countable interpolation  with an
order determining system  of states is dual to the category of Bauer
simplices whose objects are pairs $(\Omega,g),$ where $\Omega\ne
\emptyset$ is a Bauer simplex  such that $\partial_e \Omega$ is an
F-space (any two disjoint open $F_\sigma$ subsets of $\Omega$ have
disjoint closures) and $g:\Omega \to \Omega$ is a continuous
function such that $g^n=g.$ This is a Stone Duality Type Theorem,
and it generalizes the famous result by Stone \cite{Sto} that says
that the category of Boolean algebras is dual to the category of
Stone spaces (=  compact Hausdorff topological space with a base
consisting  of clopen sets).

\end{enumerate}

The paper is organized as follows. The basic properties of effect
algebras are gathered in Section 2.  The main object of our study,
state-operators and $n$-state-operators on effect algebras, are
introduced in Section 3. Our study will use facts on Choquet
simplices and their connection to effect algebras and this is
pointed out in Section 4. Section 5 characterizes
$n$-state-operators. Our first  goal, the Loomis--Sikorski Theorem
for $n$-state-operators on monotone $\sigma$-complete effect
algebras, is described in Section 6. Finally, Section 7 presents the
second main goal, it gives  Stone Type Dualities for some categories
of effect algebras and the categories of Bauer simplices whose
boundary is an F-space.

\section{Elements of Effect Algebras}

An {\it effect algebra} is by \cite{FoBe} a partial algebra $E =
(E;+,0,1)$ with a partially defined operation $+$ and two constant
elements $0$ and $1$  such that, for all $a,b,c \in E$,
\begin{enumerate}

\item[(i)] $a+b$ is defined in $E$ if and only if $b+a$ is defined, and in
such a case $a+b = b+a;$

\item[(ii)] $a+b$ and $(a+b)+c$ are defined if and
only if $b+c$ and $a+(b+c)$ are defined, and in such a case $(a+b)+c
= a+(b+c);$

\item[(iii)] for any $a \in E$, there exists a unique
element $a' \in E$ such that $a+a'=1;$

\item[(iv)] if $a+1$ is defined in $E$, then $a=0.$
\end{enumerate}

If we define $a \le b$ if and only if there exists an element $c \in
E$ such that $a+c = b$, then $\le$ is a partial ordering on $E$, and
we write $c:=b-a.$ It is clear that $a' = 1 - a$ for any $a \in E.$

For a comprehensive source on the theory of effect algebras, we
recommend \cite{DvPu}. A {\it state} on an effect algebra $E$ is any
mapping $s: \ E \to [0,1]$ such that (i) $s(1) = 1$, and (ii)
$s(a+b) = s(a) + s(b)$ whenever $a+b$ is defined in $E$.  We denote
by ${\mathcal S}(E)$ the set of all states on $E$. It can happen
that ${\mathcal S}(E)$ is empty, see e.g. \cite[Ex 4.2.4]{DvPu}.  A
state $s$ is said to be {\it extremal} if $s = \lambda s_1 +
(1-\lambda)s_2$ for $\lambda \in (0,1)$ implies $s = s_1 = s_2.$ By
$\partial_e{\mathcal S}(E)$ we denote the set of all extremal states
of ${\mathcal S}(E)$  on $E.$ We say that a net of states,
$\{s_\alpha\}$, on $E$ {\it weakly converges} to a state, $s,$ on
$E$ if $s_\alpha(a) \to s(a)$ for any $a \in E$. In this topology,
${\mathcal S}(E)$ is a compact Hausdorff topological space and every
state on $E$ lies in the weak closure of the convex hull of the
extremal states as it follows from the Krein-Mil'man Theorem,
\cite[Thm 5.17]{Goo}.

Let $G=(G;+,0)$ be an Abelian po-group (= partially ordered group).
An element $u\in G$ is said to be a {\it strong unit} if given $g\in
G,$ there is an integer $n\ge 1$ such that $g \le nu.$  If we set
$\Gamma(G,u)=[0,u]$ and  endow it  with the restriction of the group
addition, $+$,  then $\Gamma(G,u):=(\Gamma(G,u); +, 0,u)$ is an
effect algebra.

An effect algebra that is either of the form $\Gamma(G,u)$ for some
element $u\ge 0$ or is isomorphic with some $\Gamma(G,u)$  is called
an {\it interval effect algebra}.

Let $u$ be a positive element of an Abelian po-group $G.$ The
element $u$ is said to be {\it generative} if every element $g \in
G^+$ is a group sum of finitely many elements of $\Gamma(G,u),$ and
$G = G^+-G^+.$ Such an element is  a strong unit \cite[Lem
1.4.6]{DvPu} for $G$ and the couple $(G,u)$ is said to be a {\it
po-group with generative strong unit}.  For example, if $u$ is a
strong unit of an interpolation po-group $G,$ then $u$ is
generative.  The same is true for $I$ and $\mathcal E(H):= \Gamma({\mathcal
B}(H),I).$

Let $E$ be an effect algebra and $H$ be an Abelian (po-) group. A mapping
$p: E \to H$ that preserves $+$ is called an $H$-{\it valued
measure} on $E.$

\begin{Remark}\label{re:2.1} {\rm
If $E$ is an interval effect algebra, then there is a po-group $G$
with a generative strong unit $u$ such that $E \cong \Gamma(G,u),$
and every $H$-valued measure $p:\Gamma(G,u) \to H$ can be extended
to a group-homomorphism $\phi$ from $G$ into $H.$ If $H$ is a
po-group, then $\phi$ is a po-group-homomorphism. Then $\phi$ is
unique and $(G,u)$ is also unique up to isomorphism of unital (po-)
groups, see \cite[Cor 1.4.21]{DvPu}; the element $u$ is said to be a
{\it universal strong unit} for $\Gamma(G,u)$ and the couple $(G,u)$
is said to be a {\it unigroup}.}
\end{Remark}

We recall that an effect algebra  $E$ satisfies the {\it Riesz
Decomposition Property} ((RDP) in abbreviation) if $x_1 + x_2 = y_1
+ y_2$ implies there exist four elements $c_{11}, c_{12}, c_{21},
c_{22} \in E$ such that $x_1 = c_{11} + c_{12},$ $x_2 = c_{21} +
c_{22},$ $y_1 = c_{11} + c_{21},$ and $y_2 = c_{12} + c_{22}.$
Equivalently, \cite[Lem 1.7.5]{DvPu}, $E$ has (RDP) iff $x \le y_1 +
y_2$ implies that there exist two elements $x_1, x_2 \in E$ with
$x_1 \le y_1 $ and $x_2 \le y_2$ such that $x = x_1 + x_2$.

We say that an Abelian po-group $G$ is an {\it interpolation group},
if given $x_1,x_2,y_1,y_2$ in $G$ such that $x_i\leq y_j$ for all
$i,j$, there exists $z$ in $G$ such that $x_i\leq z\leq y_j$ for all
$i,j$. Equivalently, \cite[Prop 2.1]{Goo}, $G$ is an interpolation
group iff an analogous property as (RDP) for effect algebras holds
also for $G^+=\{g\in G: g\ge 0\}.$

\begin{Remark}\label{re:2.2}
{\rm (1) If $E$ is an effect algebra satisfying (RDP),  then $E$ is
an interval effect algebra. In such a case, there is a unique (up to
isomorphism of unital po-groups) interpolation unital po-group
$(G,u)$ such that $E \cong \Gamma(G,u).$ Moreover, $u$ is a
universal strong unit for $E.$ Conversely, if $(G,u)$ is an
interpolation unital po-group, then $\Gamma(G,u)$ satisfies (RDP),
and $u$ is a universal strong unit for $\Gamma(G,u),$ see \cite{Rav}
(\cite[Thm 1.7.17]{DvPu}).

(2) We note that the identity operator $I$ on a Hilbert space $H$ is
a universal strong unit for $\Gamma({\mathcal B}(H),I),$ \cite[Cor
1.4.25]{DvPu} that does not satisfy (RDP).

 }
\end{Remark}

We recall that an {\it MV-algebra} is an algebra $(A;\oplus,^*,0)$
of signature $\langle 2,1,0\rangle,$ where $(A;\oplus,0)$ is a
commutative monoid with neutral element $0$, and for all $x,y \in A$
\begin{enumerate}
\item[(i)]  $(x^*)^*=x,$
\item[(ii)] $x\oplus 1 = 1,$ where $1=0^*,$
\item[(iii)] $x\oplus (x\oplus y^*)^* = y\oplus (y\oplus x^*)^*.$
\end{enumerate}

We define also two  additional total operations $\odot$ and
$\ominus$ on $A$ via $x\odot y:= (x^*\oplus y^*)^*$ and $x\ominus y
= x\odot y^*.$

If $(G,u)$ is an Abelian $\ell$-group (= lattice ordered group) with
a strong unit $u\ge 0,$  then $(\Gamma(G,u);\oplus,^*,0)$ is a
prototypical example of an MV-algebra, where $\Gamma$ is the Mundici
functor, where $\Gamma(G,u):= [0,u]$ is endowed with the
MV-operations $g_1\oplus g_2 := (g_1 +g_2)\wedge u,$ $g^* := u-g,$
because by \cite{Mun}, every MV-algebra is isomorphic to some
$\Gamma(G,u).$

If on an MV-algebra $A$ we define a partial operation, $+$, by $a+b$
is defined in $A$ iff $a\le b^*,$ and we set then $a\oplus b:=
a\oplus b.$ Then $(A;+,0,1)$ is an interval effect algebra with
(RDP).

An {\it ideal} of an effect algebra $E$ is a non-empty subset $I$ of
$E$ such that (i) $x \in E$, $y \in I$, $x\le y$ imply $x \in I$,
and (ii) if $x,y \in I$ and $x+y$ is defined in $E$, then $x+y \in
I$.  We denote by ${\mathcal I}(E)$ the set of all ideals of $E.$ An
ideal $I$ is said to be a {\it Riesz ideal} if, for $x \in I$,
$a,b \in E$ and $x \le a+b$, there exist $a_1,b_1 \in I$ such that
$x = a_1 +b_1$ and $a_1 \le a$ and $b_1 \le b$.

For example, if $E$ has   (RDP), then any ideal of $E$ is Riesz.

We say that a poset $E$ is an {\it antilattice} if  joins and meets
exist only for comparable elements.  For example, every linearly
ordered set is an antilattice. According to \cite[Thm 2.12]{Rav} or
\cite[Thm 7.2]{Dvu2}, every effect algebra with (RDP) is a
subdirect product of antilattice effect algebras with (RDP).

For example, if $H$ is a Hilbert space, then ${\mathcal B}(H)$ is an
antilattice \cite{LuZa}, but $\mathcal E(H)=\Gamma({\mathcal B}(H),I)$ is
not. In fact, if $P_M$ and $P_N$ are orthogonal projectors onto
subspaces $M$ and $N$ of $H$, then $P_M\vee P_N$ exists in $\mathcal E(H)$
and equals $P_{M\vee N},$ \cite{Dvu0}, where $M\vee N$ denotes the join
in the complete lattice of all closed subspaces of $H,$ whereas
their join in ${\mathcal B}(H)$ fails when  they are not
comparable.

If $E=\Gamma(G,u)$ for some effect algebra with (RDP), then all
joins and meets from $E$ are the same also in $G.$

We recall that if $(G,u)$ is an Abelian unital po-group, then a {\it
state} on it is any mapping $s:G\to \mathbb R$ such that (i)
$s(g)\ge 0$ for any $g\ge 0,$ (ii) $s(g_1+g_2)=s(g_1)+(g_2)$ for all
$g_1,g_2 \in G,$ and (iii) $s(u)=1.$ A state $s$ is {\it extremal}
if from $s = \lambda s_1 +(1-\lambda)s_2$ for $\lambda \in (0,1)$ it
follows $s=s_1=s_2.$ We denote by $\mathcal{S}(G,u)$ and by
$\partial_e \mathcal S(G,u)$ the sets of all states and all extremal
states on $(G,u).$ We can also introduce the weak topology on
$\mathcal S(G,u).$  We have  that $\mathcal{S}(G,u)$ is always
nonempty, \cite[Cor 4.4]{Goo}, whenever $u>0.$ Due to the
Krein--Mil'man Theorem, \cite[Thm 5.17]{Goo}, every state on $(G,u)$
is a weak limit of  a net of convex combinations of extremal states
on $(G,u).$ If we set $\Gamma(G,u)= [0,u],$ then the restriction of
any state on $(G,u)$ onto $\Gamma(G,u)$ is a state on $\Gamma(G,u).$
We recall that if $u$ is a strong unit and $G$ is an interpolation
group,   in particular, an $\ell$-group, or more general a unigroup,
then every state on $\Gamma(G,u)$ can be uniquely extended to a
state on $(G,u).$  Moreover, this correspondence is an affine
homeomorphism (affine means that it preserves all convex
combinations).

We say that a po-group $G$ is {\it Archimedean} if for $x,y \in G$
such that $nx \le y$ for all positive integers $n \ge 1$, then $x
\le 0.$

\begin{Remark}\label{re:2.3}{\rm
It is possible to show that a unital group $(G,u)$  is
Archimedean iff $G^+=\{g \in G:\ s(g) \ge 0$ for all $ s \in
{\mathcal S}(G,u)\},$  \cite[Thm 4.14]{Goo}, or equivalently,
$\Gamma(G,u)$ has an {\it order determining} system of states,
$\mathcal S,$ i.e., $f\le g$ iff $s(f)\le s(g)$ for any $s\in
\mathcal S.$ In a similar way we define an order determining system
of states on an effect algebra $E.$  If $E=\Gamma(G,u)$ and $(G,u)$
is a unigroup, then $E$ has an order determining system of state iff
$(G,u)$ has it.  In particular, $\partial_e \mathcal S(E)$ is order
determining iff so is $\mathcal S(E).$}
\end{Remark}

An analogous result holds also for some effect algebras, see
Proposition \ref{pr:4.2} below.

\section{State-Operators and $n$-State-Operators on Effect Algebras}

In this section we introduce state-operators, $n$-state-operators,
strong state-operators, and state-morphism-operators on effect
algebras and we present their basic properties. We show that in the
case of MV-algebras, a strong state-operator and a
state-morphism-operator coincide with state-operators and
state-morphism-operators defined for MV-algebras, Propositions
\ref{pr:3.5}--\ref{pr:3.6}.

Let $E$ and $F$ be two effect algebras. A mapping $h:\ E \to F$ is
said to be a {\it homomorphism} if (i) $h(a+b) = h(a) + h(b)$
whenever $a+b$ is defined in $E$, and (ii) $h(1) =1$. In particular,
we have $h(a')=h(a)'$ for each $a\in E,$  $h(0)=0,$ $h(a)\le h(b)$
whenever $a\le b$ and then $h(b-a)=h(b)-h(a).$ A bijective
homomorphism $h$ such that $h^{-1}$ is a homomorphism is said to be
an {\it isomorphism} of $E$ and $F$.

Let $E$ be an effect algebra. An endomorphism $\tau: E \to E$ such
that  $\tau^2 =\tau$ is said to be a {\it state-operator} or an {\it
internal state} and the couple $(E,\tau)$ is said to be a {\it state
effect algebra} with internal state.

An endomorphism $\tau: E \to E$ such that
$$
\tau(\tau(a)\vee \tau(b))= \tau(a)\vee \tau(b) \eqno(3.1)
$$
whenever $\tau(a)\vee \tau(b)$ is defined in $E$ is said to be a
{\it strong state-operator} on $E,$ and the couple $(E,\tau)$ is
called a {\it strong state effect algebra}.

From (3.1) we see that, for any $a \in E,$
$\tau^2(a)=\tau(\tau(a)\vee \tau(a))= \tau(a)\vee \tau(a) =\tau(a),$
i.e., any strong state-operator is a state-operator on $E.$

If a strong-state-operator $\tau$ satisfies also  $\tau (a\vee
b)=\tau(a)\vee \tau(b)$ whenever $a\vee b$ is defined in $E,$ $\tau$
it is called a {\it state-morphism-operator} on $E,$ and the couple
$(E,\tau)$ is said to be  a {\it state-morphism effect algebra.} An
endomorphism $\tau$ is a state-morphism-operator iff $\tau^2=\tau$
and $\tau$ preserves all existing joins in $E.$

Finally, we generalize the just defined notions as follows. Given an
integer $n\ge 1,$ an endomorphism $\tau:E \to E$ is said to be an
$n$-{\it state-operator} if $\tau$ is $n$-potent, i.e. $\tau^n =
\tau.$ The couple $(E,\tau)$ is said to be an $n$-{\it state
effect algebra}. An $n$-state-operator $\tau$ is said to be an
$n$-state-morphism-operator if $\tau$ preserves all existing joins
in $E,$ and the couple $(E,\tau)$ is said to be an $n$-{\it
state-morphism effect algebra}.

We recall that a state $s$ on $E$ is {\it discrete} if there is an
integer $n\ge 1$ such that $s(E)\subseteq \{0,1/n,\ldots, n/n\}.$

We have that if $s \in \mathcal S(E),$ then $s\circ \tau \in
\mathcal S(E),$ and if $s$ is discrete, so is $s\circ \tau.$

We say that a state-operator $\tau$ satisfies the {\it extremal
state property}, (ESP) for short, if  $s\circ \tau \in
\partial_e \mathcal S(E),$ for any $s\in \partial_e \mathcal S(E),$
and we say also that $(E,\tau)$ satisfies (ESP).

Let $\tau$ be an endomorphism on $E.$ We denote by
$$\Ker(\tau)=\{a\in E: \tau(a)=0\}$$
the {\it kernel} of $\tau.$ An endomorphism $\tau$ is {\it faithful}
if $\Ker(\tau)=\{0\}.$  An ideal $I$ of $E$ is said to be a
$\tau${\it -ideal} if $\tau(I) \subseteq I.$  For example,
$\Ker(\tau)$ is a $\tau$-ideal.

\begin{Example}\label{ex:3.1} {\rm
(1) The couple $(E,\id_E)$ is a state-morphism effect algebra with
(ESP).

(2) Let $F$ be an effect algebra and let $E = F \times F.$ We define
two operators on $E$ by
$$ \tau_1(a,b) =(a,a), \quad \mbox{and}\quad \tau_2(a,b)=
(b,b),\quad (a,b) \in F\times F.
$$
Then $\tau_1$ and $\tau_2$ are state-morphism-operators on $E$ that
preserve all joins and meets existing in $E$:

In fact, if $\mathcal \partial_e \mathcal S(F)=\emptyset,$ then $\partial_e
\mathcal S(E)=\emptyset$ and both $\tau_1$ and $\tau_2$ trivially satisfy
(ESP).

Assume that $\partial_e \mathcal S(F)=\{s_t: t \in T\}$ for some
index set $T\ne \emptyset.$  Define $m_t^1(a,b)=s_t(a)$ and $m_t^2(a,b)=s_t(b)$
for all $(a,b) \in F \times F$ and each $t\in T.$ If $m$ is an
extremal state on $E,$ then either $m(1,0)=1$ or $m(0,1)=1.$ In the
first case, $s(a):= m(a,0)$ is a state on $F.$ It is extremal,
otherwise, $s(a)= \lambda_1 s_1(a)+\lambda_2 s_2(a)$ for some states
$s_1,s_2$ on $E.$  If we define $m_i(a,b):= s_i(a),$ then $m_1,m_2$
are states on $E$ and $m(a,b)= \lambda_1 m_1(a,b)+\lambda_2m_2(a,b)$
which is impossible, hence $s(a)$ is an extremal state on $F$ so
that $s=s_t$ for some $t\in T,$ and $m(a,b)= m(a) =s_t(a) =
m_t^1(a,b).$

Similarly, in the second case, $m= m_t^2$ for some $t\in T.$
Consequently, $\partial_e \mathcal S(E)=\{m_t^j: t\in T, j=1,2\}.$

Check:  $m_t^1(\tau_1(a,b)) = m_t^1(a,a)= s_t(a)=m_t^1(a,b)$ and
$m_t^2(\tau_1(a,b)) = m_t^2(a,a) = s_t(a)=m_t^1(a,b).$  Therefore,
$\tau_1$  as well as $\tau_2$ are state-morphism-operators with
(ESP). }
\end{Example}

\begin{Lemma}\label{le:3.2}
Let $\tau$ be an  endomorphism of an effect algebra $E$ such that
$\tau^2 = \tau.$ Then

\begin{enumerate}

\item[{\rm (i)}] If  $\tau$ is a strong state-operator, then $\tau$
preserves all existing meets from $E$  of the form $\tau(a)\wedge
\tau(b).$

\item[{\rm (ii)}] The set $\tau(E)$ is an effect subalgebra, $\tau(E) =
\{a\in E: \tau(a)=a\},$ and $\tau$ on $\tau(E)$ is the identity on
$\tau(E).$ If $\tau$ is also a strong state-operator, then if
$\tau(a)\vee \tau(b) \in E,$ then $\tau(a)\vee \tau(b)\in \tau(E).$

\item[{\rm (iii)}]  If $E$ satisfies {\rm (RDP)}, then so does $\tau(E).$

\item[{\rm (iv)}]  If $\tau$ is faithful, then $a< b$ entails
$\tau(a)<\tau(b).$

\item[{\rm (v)}] If $\tau$ is faithful, then either $\tau(a)=a$ or
$\tau(a)$ and $a$ are not comparable.

\item[{\rm (vi)}]  If $E$ is linear and $\tau$ faithful, then
$\tau(a) = a$ for any $a \in E.$

\item[{\rm (vii)}] If $E$ is an antilattice effect algebra, then $\tau$
preserves all existing meets and joins.

\item[{\rm (viii)}]
If $\tau:E\to E$ is  faithful  then $\tau$ is a strong
state-operator.
\end{enumerate}

\end{Lemma}

\begin{proof}
(i)    Passing to negation, we see that $\tau$ preserves all
existing meets in $\tau(a)\wedge \tau(b) \in E.$

(ii) Since $a \in \tau(E)$ iff $a = \tau(b)$ for some $b \in E,$ we
have $\tau(a)=\tau(\tau(b))= \tau(b) = a.$

Let $\tau(a)+\tau(b)$ be defined in $E.$  Then
$\tau(\tau(a)+\tau(b))= \tau(\tau(a))+ \tau(\tau(b)) =
\tau(a)+\tau(b) \in \tau(E).$ It is clear now that the restriction
of $\tau$ onto $\tau(E)$ is the identity on $\tau(E).$

Now if $\tau$ is a strong state-operator, the statement follows from
(3.1).

(iii)  Let $E$ satisfy (RDP) and let $\tau(a_1)+\tau(a_2)=
\tau(b_1)+\tau(b_2).$  There are four elements $c_{11}, c_{12},
c_{21}, c_{22} \in E$ such that $\tau(a_1)= c_{11} + c_{12},$
$\tau(a_2) = c_{21} + c_{22},$ $\tau(b_1)= c_{11} +c_{21}$ and
$\tau(b_2)= c_{12} + c_{22}.$ Then $\tau(a_1)= \tau(\tau(a_1)) =
\tau(c_{11} + c_{12}) = \tau(c_{11}) + \tau(c_{12}),$ similarly for
$\tau(a_2), \tau(b_1), \tau(b_2)$ proving that the elements
$\tau(c_{11}),\tau(c_{12}), \tau(c_{21}), \tau(c_{22}) \in \tau(E)$
yield that $\tau(E)$ satisfies (RDP).

(iv)  Suppose that $a<b$ and $\tau(a)=\tau(b).$  Then $\tau(b-a) =
\tau(b)-\tau(a)= \tau(a)-\tau(a)=0$ giving $b-a=0$ so that $a=b,$ a
contradiction.

(v)  Assume that $\tau(a)\ne a$ and let $\tau(a)$ and $a$ be
comparable.  Then either $a< \tau(a)$ or $\tau(a) <a.$  By (iv) we
have $\tau(a) < \tau(a)$ that is impossible.

(vi) It follows directly from (v).

(vii)  Let $a\vee b$ be defined in $E.$  Then $a$ and $b$ are
comparable. So are $\tau(a)$ and $\tau(b)$ so that $\tau(a\vee b)=
\tau(a)\vee \tau(b).$ Going to negations, we see that $\tau$
preserves also meets.

(viii)   Assume that $d=\tau(a)\vee \tau(b)$ is defined in $E.$ We
show that $\tau(d)=d.$ Check, $\tau(d)\ge \tau(\tau(a))= \tau(a)$
and $\tau(d)\ge \tau(\tau(b))= \tau(b).$ This yields $\tau(d)\ge
\tau(a)\vee \tau(b)=d.$ Hence, we have
$\tau(\tau(d)-d)=\tau(\tau(d))-\tau(d)=0, $ i.e. $\tau(d)-d=0$ and
$\tau(d)=d.$
\end{proof}

\begin{Proposition}\label{pr:3.3}
Let $E=\Gamma(G,u)$ be an interval effect algebra, where $(G,u)$ is
a unital po-group with universal strong unit for $E.$ Let $n$ be a
fixed integer.

Every endomorphism $\tau $ with $\tau^n=\tau$ on $E$ can be uniquely
extended to an $n$-potent po-group homomorphism $\tau _u$ on
$(G,u),$ i.e. $\tau_u^n = \tau_u.$ Conversely, the restriction of
any $n$-potent po-group homomorphism of $(G,u)$ to $E$ gives an
$n$-potent endomorphism on $E.$

\end{Proposition}

\begin{proof}   Let $\tau $ be an $n$-potent  endomorphism on $E$. Since
$E=\Gamma(G,u),$  the mapping $\tau: E\to E$ is in fact a $G$-valued
measure on $E.$ By Remark \ref{re:2.1}, there is a unique extension,
$\tau_u:G \to G$ of $\tau$ that is a po-group-homomorphism.

Now we show that $\tau_u^n =\tau_u.$  Since every element $x \in
G^+$ is expressible via $x = x_1+\cdots+x_n,$ where
$x_1,\ldots,x_n\in E,$ we have $\tau_u^n(x)= \tau_u^n(x_1)+\cdots+
\tau_u^n(x_n)= \tau^n(x_1)+\cdots+ \tau^n(x_n) = \tau(x_1)+\cdots+
\tau(x_n) = \tau_u(x_1)+\cdots+ \tau_u(x_n).$ Every element $x \in
G$ is of the form $x = x_1- x_2,$ where $x_1,x_2\in G^+,$ then
$\tau_u^n(x)= \tau_u^n(x_1)-\tau_u^n(x_2)= \tau_u(x_1)-\tau_u(x_2)=
\tau_u(x).$
\end{proof}

\begin{Proposition}\label{pr:3.4} If $E$ is a linear effect
algebra with {\rm (RDP)}, then every endomorphism on $E$ preserves
joins and if $s$ is an extremal state, then $s \circ \tau$ is an
extremal state on $E.$
\end{Proposition}

\begin{proof}  Suppose that $E =\Gamma(G,u)$ for some unital po-group $(G,u).$
Then $(G,u)$ is a unital $\ell$-group and a state $s$ on $E$ is
extremal iff its extension on $(G,u)$, denoted also by $s,$ is
extremal.  By \cite[Thm 12.18 ]{Goo}, $s$ is extremal iff $s(g\wedge
h)=\min\{s(g),s(h)\}$ for all $g,h \in G^+$ iff $s(g\wedge
h)=\min\{s(g),s(h)\}$ for all $g,h \in E.$

Since $E$ is linear, $\tau$ trivially preserves all joins and meets
in $E.$  Hence, if $s$ is an extremal state on $E,$ then
$s\circ\tau((a\wedge b)) = s(\tau(a)\wedge \tau(b))=
\min\{s(\tau(a)), s(\tau(b))\}$ proving that $s\circ \tau $ is an
extremal state on $E.$
\end{proof}

According to \cite{FlMo}, a \emph{state MV-algebra}  is a couple
$(A,\tau),$ where $A$ is an MV-algebra   and $\tau$ is a unary
operator on $A$ (an internal state or an {\it MV-state-operator})
satisfying for each $x,y\in A$:\\
\\
$(1)_{MV}$\ \ $\tau(0)=0,$\\
$(2)_{MV}$\ \ $\tau(x^{*})=(\tau(x))^{*},$\\
$(3)_{MV}$\ \ $\tau(x\oplus y)=\tau(x)\oplus\tau(y\odot(x\odot
y)^{*}),$\\
$(4)_{MV}$\ \ $\tau(\tau(x)\oplus \tau(y))=\tau(x)\oplus\tau(y).$\\

In \cite{FlMo} it is shown that for any state MV-algebra we have (i)
$\tau(\tau (x))=\tau (x)$, (ii) $\tau(1)=1$, (iii) if $x\leqslant
y$, then $ \tau(x)\leqslant\tau(y),$ (iv) $\tau(x\oplus y)\le
\tau(x)\oplus \tau(y),$ (v)  $\tau(x+y) = \tau(x)+\tau(y),$ (vi)
$\tau(\tau(x)\vee \tau(y))=\tau(x)\vee \tau(y),$ and (vii) the image
$\tau(A)$ is the domain of an MV-subalgebra of $A$ and
$(\tau(A),\tau)$ is a state MV-subalgebra of $(A,\tau).$

According to \cite{DiDv}, an {\it MV-state-morphism-operator} on an
MV-algebra $A$ is any endomorphism $\tau: A \to A$ such that $\tau^2
= \tau.$  Every MV-state-morphism-operator is a state-operator.

\begin{Proposition}\label{pr:3.5} Let $A$ be an MV-algebra.
If $\tau$ is an MV-state-operator  on the MV-algebra $A,$ then
$\tau$ is a strong state-operator on  $A$ taken as an effect algebra
with the partial addition $+$ derived from $\oplus.$

Conversely, if $\tau$ is a strong state-operator on the derived
effect algebra $A,$ then $\tau$ is an MV-state-operator on the
MV-algebra $A.$
\end{Proposition}

\begin{proof}  Let $\tau$ be an MV-state-operator on the MV-algebra
$A.$  Due to the basic properties of $\tau$ described just before
this proposition, we see that $\tau$ is a strong state-operator on
the effect algebra $A$ derived from the MV-algebraic structure.
Moreover, $\tau$ preserves all joins and meets in $\tau(A).$

Conversely, let $\tau$ be a strong state-operator on the effect
algebra $A.$  Then $(1)_{MV}$ and $(2)_{MV}$ hold.

We have $x\oplus y = x + ((x\oplus y)\ominus x) = x + (y\wedge x^*)$
so that $\tau(x\oplus y) = \tau(x) + \tau(y \wedge x^*)=
\tau(x)\oplus\tau(y\odot(x\odot y)^{*})$ that gives $(3)_{MV}.$

In addition, $\tau(\tau(x)\oplus \tau(y)) = \tau(\tau(x) +
\tau(y)\wedge \tau(x)^*) = \tau(\tau(x) + \tau(y) \wedge \tau(x^*))
= \tau(\tau(x)) + \tau(\tau(y) \wedge \tau(x^*)) = \tau(x) + \tau(y)
\wedge \tau(x^*) = \tau(x) \oplus \tau(y)$ and this proves
$(4)_{MV}.$
\end{proof}

\begin{Proposition}\label{pr:3.6} Let $A$ be an MV-algebra.
If $\tau$ is an MV-state-morphism-operator  on the MV-algebra $A,$
then $\tau$ is a state-morphism-operator with {\rm (ESP)} on  $A$
taken as an effect algebra with the partial addition $+$ derived
from $\oplus.$

Conversely, if $\tau$ is a state-morphism-operator on the effect
algebra derived from an  MV-algebra $A,$ then $\tau$ is an
MV-state-morphism-operator on the MV-algebra $A.$ In addition,
$\tau$ is with {\rm (ESP)} on the effect algebra $A.$
\end{Proposition}

\begin{proof} Let $\tau$ be an MV-state-morphism-operator on $A.$
Due to Proposition \ref{pr:3.5}, $\tau$ is a strong state-operator
on the effect algebra $A$ that preserves all joins and meets.

Let $s$ be an extremal state, that $s\circ \tau$ is also an extremal
state because $s\circ \tau(a\wedge b) = s(\tau(a\wedge b)) =
s(\tau(a)\wedge \tau(b)) = \min\{s(\tau(a)),s(\tau(b))\}$ for all
$a,b \in A.$

Conversely, let $\tau$ be a state-morphism-operator on the effect
algebra $A.$ Then $\tau$ preserves all joins and meets in $A,$ so
that $\tau$ is an MV-state-morphism-operator on the MV-algebra $A.$
Due to the first part of the present proof, $\tau$ satisfies (ESP)
on the effect algebra $A.$
\end{proof}

\section{Effect-Clans and Choquet Simplices}

In this section, we show a close connection between the state spaces
of effect algebras and Choquet simplices.

An {\it effect-clan} is a system  ${\mathcal E}$  of $[0,1]$-valued
functions on $\Omega\ne \emptyset$ such that (i) $1 \in {\mathcal
E}$, (ii) $f \in {\mathcal E}$ implies $1-f \in {\mathcal E}$, and
(iii) if $f,g \in {\mathcal E}$ and $f(\omega) \le 1 -g(\omega)$ for
any $\omega \in \Omega$, then $f+g \in {\mathcal E}$. Then the effect-clan
${\mathcal E}$ is an effect algebra that
is not necessarily a Boolean algebra nor an MV-algebra.

If $E$ is an effect-clan of characteristic functions on $\Omega$,
then $E$ satisfies (RDP) iff $E_0=\{A\subseteq \Omega: \chi_A \in
E\}$ is an algebra of subsets of $\Omega.$  For example, if $\Omega$
is a finite set with an even number of elements, the set of all
characteristic functions of subsets of $\Omega$  with an even number of
subsets is an effect-clan where (RDP) fails:
$\chi_{\{1,2\}},\chi_{\{1,3\}} \in E$ but $\chi_{\{1\}} \notin E.$

\begin{Proposition}\label{pr:4.2}  Let $E=\Gamma(G,u),$ where
$(G,u)$ is a unigroup. The following statements are equivalent:

\begin{enumerate}
\item[{\rm (i)}]   ${\mathcal S}(E)$ is order determining.

\item[{\rm (ii)}] $E$ is isomorphic to some effect-clan.

\item[{\rm (iii)}] $G$ is Archimedean.

\end{enumerate}

\end{Proposition}

\begin{proof} (i) $\Rightarrow$ (ii). Given $a\in E$, let $\hat a$ be a
function from ${\mathcal S}(E)$ into the real interval $[0,1]$ such
that $\hat a(s):=s(a)$ for any $s \in {\mathcal S}(E),$ and let
${\hat E}=\{\hat a:\ a \in E\}.$ We endow $\hat E$ with pointwise
addition, so that $\hat E$ is an effect-clan. Since ${\mathcal
S}(E)$ is order determining, the mapping $a\mapsto \hat a$ is an
isomorphism and $\hat E$ is Archimedean.

(ii) $\Rightarrow$ (iii).  Let $\hat E$ be any effect-clan
isomorphic with $E,$ and let $a\mapsto \hat a$ be such an
isomorphism. Let $G(\hat E)$ be the po-group generated by $\hat E.$
Then $G(\hat E)$ consists of all functions of the form $\hat
a_1+\cdots+\hat a_n -\hat b_1-\cdots- \hat b_m,$ and $\hat 1$ is its
strong unit.   Due to the categorical equivalence, $(G,u)$ and
$(G(\hat E),\hat 1)$ are isomorphic. But $(G(\hat E),\hat 1)$ is
Archimedean, so is $(G,u).$

(iii) $\Rightarrow$ (i). Due to \cite[Thm 4.14]{Goo}, $G^+=\{g \in
G:\ s(g) \ge 0$ for all $ s \in {\mathcal S}(G,u)\},$ which means
that ${\mathcal S}(G,u)$ is order determining. Hence, the
restrictions of all states on $(G,u)$ onto $E=\Gamma(G,u)$ imply
${\mathcal S}(E)$ is order determining.
\end{proof}

We say a poset $E$ is {\it monotone $\sigma$-complete} provided that
for every ascending (descending) sequence $x_1 \le x_2 \le \cdots $
($x_1 \ge x_2 \ge \cdots$) in $E$ which is bounded above (below) in
$E$ has a supremum (infimum) in $E.$

An {\it effect-tribe} on a set $\Omega \ne \emptyset$ is any system
${\mathcal T} \subseteq [0,1]^\Omega$ such that (i) $1 \in {\mathcal
T}$, (ii) if $f \in {\mathcal T},$ then $1-f \in {\mathcal T}$,
(iii) if $f,g \in {\mathcal T}$, $f \le 1-g$, then $f+g \in
{\mathcal T},$ and (iv) for any sequence $\{f_n\}$ of elements of
${\mathcal T}$ such that $f_n \nearrow f$ (pointwise), then $f \in
{\mathcal T},$ i.e. if $f_n(\omega) \nearrow f(\omega)$ for every
$\omega \in \Omega,$ then $f \in {\mathcal T}.$ It is evident that
any effect-tribe is a monotone $\sigma$-complete effect algebra.

Now let $\Omega$ be a compact Hausdorff space. Then
$\mbox{C}(\Omega),$ the system of all real-valued  continuous
functions on $\Omega,$  is an $\ell$-group  (we recall that, for
$f,g \in \mbox{C}(\Omega),$ $f \le g$ iff $f(x)\le g(x)$ for any $x
\in \Omega$), and it is Dedekind $\sigma$-complete iff $\Omega$ is
basically disconnected (the closure of every open $F_\sigma$ subset
of $\Omega$ is open), see \cite[Lem 9.1]{Goo}.  In such a case,
$$C_1(\Omega):=\Gamma(\mbox{C}(\Omega),1_\Omega)
$$
is a  $\sigma$-complete MV-algebra with respect to the MV-operations
that are defined by points.

Let $\Omega$ be a convex subset of a real vector space $V.$ A point
$x\in \Omega$ is said to be {\it extreme} if from $x= \lambda
x_1+(1-\lambda)x_2,$ where $x_1,x_2 \in \Omega$ and $0<\lambda <1$
we have $x=x_1=x_2.$ By $\partial_e \Omega$ we denote the set of
extreme points of $\Omega.$

Let $\Omega$ and $\Omega_1$ be convex spaces. A mapping $f:\ \Omega
\to \Omega_1$ is said to be {\it affine} if, for all $x,y \in
\Omega$ and any $\lambda \in [0,1]$, we have $f(\lambda x
+(1-\lambda )y) = \lambda f(x) +(1-\lambda ) f(y)$.

Given a compact convex set $\Omega\ne \emptyset$ in a topological
vector space, we denote by $\mbox{Aff}(\Omega)$ the collection of
all real-valued affine continuous functions on $\Omega$. Of course,
$\mbox{Aff}(\Omega)$ is a po-subgroup of the po-group
$\mbox{C}(\Omega)$ of all continuous real-valued functions on
$\Omega$, hence it is an
Archimedean unital po-group with the strong unit $1$.

We recall that a {\it convex cone} in a real linear space $V$ is any
subset $C$ of  $V$ such that (i) $0\in C,$ (ii) if $x_1,x_2 \in C,$
then $\alpha_1x_1 +\alpha_2 x_2 \in C$ for any $\alpha_1,\alpha_2
\in \mathbb R^+.$  A {\it strict cone} is any convex cone $C$ such
that $C\cap -C =\{0\},$ where $-C=\{-x:\ x \in C\}.$ A {\it base}
for a convex cone $C$ is any convex subset $\Omega$ of $C$ such that
every non-zero element $y \in C$ may be uniquely expressed in the
form $y = \alpha x$ for some $\alpha \in \mathbb R^+$ and some $x
\in \Omega.$

We recall that in view of \cite[Prop 10.2]{Goo}, if $\Omega$ is a
non-void convex subset of $V,$ and if we set
$$
C =\{\alpha x:\ \alpha \in \mathbb R^+,\ x \in \Omega\},
$$
then $C$ is a convex cone in $V,$ and $\Omega$ is a base for $C$ iff
there is a linear functional $f$ on $V$ such that $f(\Omega) = 1$
iff $\Omega$ is contained in a hyperplane in $V$ which misses the
origin.

Any strict cone $C$ of $V$ defines a partial order $\le_C$ via $x
\le_C y$ iff $y-x \in C.$ It is clear that $C=\{x \in V:\ 0 \le_C
x\}.$ A {\it lattice cone} is any strict convex cone $C$ in $V$ such
that $C$ is a lattice under $\le_C.$

A {\it simplex} in a linear space $V$ is any convex subset $\Omega$
of $V$ that is affinely isomorphic to a base for a lattice cone in
some real linear space. A  simplex $\Omega$ in a locally convex
Hausdorff space is said to be (i) {\it Choquet} if $\Omega$ is
compact, and (ii) {\it Bauer} if $\Omega$ and $\partial_e \Omega$
are compact.

Choquet and Bauer simplices are very important for our study because
(i) if $E$ is with (RDP), then $\mathcal S(E)$ is a Choquet simplex,
\cite[Thm 10.17]{Goo};  if $\Omega$ is a convex compact subset of a
locally convex Hausdorff space, then (ii) $\Omega$ is a Choquet
simplex iff $(\mbox{Aff}(\Omega),1)$ is an interpolation po-group,
\cite[Thm 11.4]{Goo}, (iii) $\mathcal S(E)$ is a Bauer simplex
whenever $E$ is an MV-algebra (Example \ref{ex:clan} below gives an
effect algebra with (RDP) that is not MV-algebra), and (iv) $\Omega$
is a Bauer simplex iff $(\mbox{Aff}(\Omega),1)$ is an $\ell$-group,
\cite[Thm 11.21]{Goo}.

Let  $\mathcal S(E)\ne \emptyset.$ Given $a \in E,$ we define the
mapping $\hat a:\mathcal S(E) \to [0,1]$  by $\hat a(s):= s(a),$ $s
\in \mathcal S(E),$  and let  $\widehat E:=\{\hat a: a \in A\}.$
Then $\widehat E$ is an effect-clan if, e.g.,  $\mathcal S(E)$ is
order determining (see Remark \ref{re:2.3} and Proposition
\ref{pr:4.2}; in such a case, $E$ and $\hat E$ are isomorphic), or
$E$ is an MV-algebra, or $E$ is a monotone $\sigma$-complete effect
algebra with (RDP), see Theorem \ref{th:5.2} below, and in these
cases the natural mapping $\psi(a):= \hat a,$ ($a \in E$) is a
homomorphism of $E$ onto $\widehat E.$

In general, $\widehat E$ is not necessarily an effect-clan:

\begin{Example}\label{ex:clan}
There is an effect algebra with {\rm (RDP)} such that $\widehat E$
is not an effect-clan.
\end{Example}

\begin{proof}
Let $\mathbb Q$ be the set of all rational numbers and let
$G=\mathbb Q\times \mathbb Q$ be ordered by the strict ordering,
i.e. $(g_1,g_2)\le (h_1,h_2)$ iff $g_1 < h_1$ and $g_2< h_2$ or
$g_1=h_1$ and $g_2=h_2.$ If we set $u=(1,1),$ $(G,u)$ is a unital
po-group with interpolation.

If we set $s_0(g,h):=h$ and $s_1(g,h):=g,$ then $s_0$ and $s_1$ are
states on $(G,u).$  We claim $\partial_e \mathcal S(G,u)
=\{s_0,s_1\}.$

Let $s \in \partial_e \mathcal S(G,u).$  Since $s(1,1)=1,$ we have
$1=s(\frac{k}{k},\frac{k}{k}) = k s(\frac{1}{k},\frac{1}{k}),$ i.e.,
$s(\frac{1}{k},\frac{1}{k})= \frac{1}{k}$ for any integer $k \ge 1.$
Hence, $s(g,g)=g$ for any $g \in \mathbb Q^+$ and consequently for
any $g \in \mathbb Q.$

Now let $g<h$ be two rational numbers. Take a sequence of rational
numbers $\{\delta_i\} \searrow 0.$ Then

\begin{eqnarray*}
s(g-\delta_i, g-\delta_i)\le &s(g,h)& \le s(h+\delta_i,h+\delta_i),\\
g-\delta_i \le &s(g,h)& \le h-\delta_i,
\end{eqnarray*}
so that
$$g\le s(g,h)\le h.
$$
In a similar way, $g \le s(h,g)\le h.$
Therefore, $1=s(1,1)= s(1,0) + s(0,1)$ and $\lambda:=s(1,0)$ and
$1-\lambda = s(0,1)$ are positive numbers.

If $\lambda =0,$ then $s(1/k,0)=0$ for each integer $k\ge 1,$ so
that $s(g,0)=0$ for any $g \in \mathbb Q^+$ and for any $g \in
\mathbb Q.$ Hence, $s(g,h) = s(0,h) = h = s_0(g,h).$  In a similar
way, if $\lambda = 1,$ then $s(g,h)=g=s_1(g,h).$  If $0<\lambda <1,$
we define $s_0'(g,h) = s(0,h)/\lambda$ and
$s_1'(g,h):=s(g,0)/(1-\lambda),$ $s_0'$ and $s_1'$ are states on
$(G,u)$ such that $s_0'=s_0$ and $s_1'=s_1$ and $s =
s_\lambda:=\lambda s_0 +(1-\lambda)s_1,$ $\lambda \in [0,1],$ so
that $\partial_e \mathcal S(G,u)=\{s_0,s_1\}.$

Let us set $E=\Gamma(G,u).$  Then $E$ is an effect algebra with
(RDP), and $\mathcal S(E) =\{s'_\lambda:\lambda \in [0,1]\}$ where
$s'_\lambda$ is the restriction of $s_\lambda$ to $E.$ We next
define $\widehat E= \{\hat a: a\in E\}.$  We assert that this
$\widehat E$ is not an effect-clan.

Indeed, let $a=(0.3,0.3),$ $b=(0.7,0.4).$  Then $\hat a(s'_\lambda)=
0.3\lambda +0.3(1-\lambda) =0.3  \le 1-\hat b(s'_\lambda) =
0.3\lambda +0.6(1-\lambda)$ for any $\lambda\in [0,1],$ but there is
no $c \in E$ such that $\hat c = \hat a + \hat b$ because $\hat
a(s_0')+ \hat b(s_0')=1$ and $\hat a(s_1')+\hat b(s_1')=0.7.$
\end{proof}

We note that the former example shows that $\mathcal S(E)$ is a
Bauer simplex that is not order determining but it is {\it
separating}, i.e. $s(a)=s(b)$ for any state $s$ on $E$ implies
$a=b.$ In addition, $(G,u)$ is not Archimedean, see Proposition
\ref{pr:4.2}.

Let us define

$$
A(\mathcal S(E)):= \Gamma(\Aff(\mathcal S(E)),1).
$$

Hence, $A(\mathcal S(E))$ is an effect algebra with (RDP) whenever
$\mathcal S(E)$ is a Choquet simplex, in particular  $E$ satisfies
(RDP), and $E$ can be converted into an MV-algebra when $\mathcal
S(E)$ is a Bauer simplex.

Nevertheless not every state space $\mathcal S(E)$ is a Bauer
simplex, we recall that according to a delicate result of Choquet
\cite[Thm I.5.13]{Alf}, $\partial_e {\mathcal S}(E)$ is always a
Baire space in the relative topology induced by the topology of
${\mathcal S}(E)$, i.e. the Baire Category Theorem holds for
$\partial_e {\mathcal S}(E).$

The following lemma  \cite[Lem 7]{BSW}, \cite[Cor 3]{Wri} will play
a crucial role in our investigation for the Loomis--Sikorski
Theorem.

\begin{Lemma}\label{le:4.1}
Let $\{a_i\}$ be  a monotone descending  sequence of nonnegative
functions in $\mbox{\rm Aff}(\Omega)$, where $\Omega$ is a convex
compact set and let $a(x) = \lim_i a_i(x)$ for any $x \in \Omega$.
Then $\bigwedge_i a_i = 0$ in $\mbox{\rm Aff}(\Omega)$ if and only
if $\{x \in \partial_e \Omega:\ a(x) > 0\}$ is a meager subset in
the relative topology of $\partial_e \Omega.$
\end{Lemma}

\section{Characterization of $n$-State-Operators on  Effect Algebras} 

We remind the reader that a state $s$ on $E$ is {\it discrete} if there is an
integer $n\ge 1$ such that $s(E)\subseteq \{0,1/n,\ldots, n/n\}.$

\begin{Proposition}\label{pr:5.1}  Let $E$ be an effect algebra with
$\mathcal S(E)\ne \emptyset$ and let $\tau$ be an $n$-state-operator
on $E.$ Then there is an affine continuous function $g:\mathcal S(E)
\to \mathcal S(E)$ such that $g^n =g,$ and $g(s)(E) \subseteq s(E)$
for any discrete state $s \in \mathcal S(E).$

Let
$$A(E) =\{f\in \Gamma(\Aff({\mathcal  S}(E)),1): f(s) \in s(E) \
\mbox{for all discrete}\ s \in \partial_e{\mathcal
S}(E)\}.\eqno(4.1)
$$
Then $A(E)$ is an effect-clan and the mapping $\tau_g:A(E) \to A(E)$
defined by $\tau_g(f) = f \circ g,$ $f \in A(E),$ is an
$n$-state-operator on $A(E).$

Suppose that $\widehat E$ is an effect algebra. If we define
$\widehat \tau$ as a mapping from  $\widehat E$ into itself such
that $\widehat\tau (\widehat a):=\widehat {\tau(a)}$ $(a\in E),$
then $\widehat \tau$ is a well-defined $n$-state-operator on
$\widehat E$ that is the restriction of $\tau_g.$

Conversely, if $g:\mathcal S(E)\to \mathcal S(E)$ is an arbitrary
affine and continuous function such that $g^n  = g,$  and $g(s)(E)
\subseteq s(E)$ for any discrete state $s \in \mathcal S(E),$ then
the mapping $\tau_g: A(E) \to A(E),$ defined by $\tau_g(f):= f\circ
g,$ $f \in A(E),$ is an $n$-state-operator.
\end{Proposition}

\begin{proof}
If $s \in \mathcal S(E),$  then $s\circ \tau \in \mathcal S(E).$
Therefore, the mapping $g: \mathcal S(E) \to \mathcal S(E)$ defined
by $g(s)= s\circ \tau,$ $s \in \mathcal S(E),$ is a well-defined
mapping and affine.

Moreover,  $g$ is continuous because if $s_\alpha \to s$, then we
have $\lim_\alpha g(s_\alpha)(a)=\lim_\alpha s_\alpha(\tau(a))
=s(\tau(a))= g(s)(a)$ for any $a \in E.$

>From the construction of $g$ we have $g^n = g$ because if $n=1,$
this is clear and if $n\ge 2,$ $g^n(s)=g^{n-1}(g(s))=g^{n-1}(s\circ
\tau)= s\circ\tau^n = s\circ \tau= g(s).$

Let $s$ be a discrete state on $E.$ Then $s(E) \subseteq
\{0,1/n,\ldots,n/n\}$ for some $n\ge 1$ and whence   $s(\tau(E))
\subseteq \{0,1/n,\ldots,n/n\}.$

It is clear that $A(E)$ is an effect-clan. Take $f \in A(E).$ Then
$f$ is a continuous function taking values in the interval $[0,1].$
To verify that $\tau_g(f) \in A(E)$ we have to show that
$\tau_g(f)(s)\in s(E)$ for any discrete extremal state $s$ on $E.$
Check: $\tau_g(f)(s)=f(g(s))=f(s\circ\tau) \in (s\circ \tau)(E))
\subseteq s(E)$ due to the just above proved statement. Hence,
$\tau_g(f)$ is again an element of $A(E).$ It is easy to verify that
$\tau_g$ is an $n$-potent endomorphism from $A(E)$ into itself.

Now we show that $\widehat \tau$ is a well-defined operator on
$\widehat E.$  Assume $\hat a = \hat b.$ This means $s(a)=s(b)$ for
any $s \in  \mathcal S(E)$.  Hence, $s(\tau(a)) = g(s)(a) = g(s)(b)=
s(\tau(b)),$ so that $\widehat {\tau (a)}= \widehat {\tau(b)}$  and
finally $\widehat \tau(\hat a)= \hat a \circ g =\hat b \circ g=
\widehat \tau (\hat b).$ Since $\widehat E$ is a subalgebra of
$A(E),$ $\widehat \tau$ is the restriction of $\tau_g.$

Finally, let $g:\mathcal S(E) \to \mathcal S(E)$ be an affine and
continuous function such that $g^n=g,$ $g(s)(E) \subseteq s(E)$ for
any discrete state $s \in \mathcal S(E).$  If $f \in A(E),$ then
$\tau_g(f):=f\circ g \in \Aff(\mathcal S(E)).$ If $s$ is an extremal
discrete state, then $f(g(s))   \in g(s)(E) \subseteq s(E)$ for any
discrete state $s\in \partial_e \mathcal S(E)$ so that $\tau_g(f)\in
A(E)$ and $\tau_g$ is an $n$-state-operator on $A(E).$
\end{proof}

The following principal representation theorem for monotone
$\sigma$-complete effect algebras with (RDP) follows from \cite[Cor
16.15]{Goo} and using the Ravindran representation theorem
\cite{Rav}, see also Remark \ref{re:2.2}.

\begin{Theorem}\label{th:5.2} Let $E$ be a nontrivial  monotone
$\sigma$-complete effect algebra with {\rm (RDP)}.  Then $E$ is
isomorphic with $A(E)$ defined  by $(4.1)$ and $\widehat E=A(E).$
\end{Theorem}

We say that an effect algebra $E$ is {\it weakly divisible}, if
given an integer $n\ge 1,$ there is an element $v\in E$ such that
$n\cdot v:= v + \cdots + v = 1.$  In such a case, $E$ has no
extremal discrete state. We notice that according to (4.1), if $E$
is a weakly divisible effect algebra that is monotone $\sigma$-complete,  it
has no discrete extremal state, therefore, $E$ is {\it divisible},
that is, given $a\in E$ and $n\ge 1,$ there is an element $v\in E$
such that $n\cdot v= a.$ Consequently, for monotone
$\sigma$-complete effect algebras, $E,$ with (RDP) the notions of
 weak divisibility and  divisibility, as well as the property
that $E$ admits no discrete (extremal) state, coincide.

We say that an $n$-state-operator $\tau$ on an effect algebra $E$ is
{\it monotone} $\sigma$-{\it complete} if whenever $a_i \nearrow a,$
that is $a_i\le a_{i+1}$ for any $i\ge 1$ and $a=\bigvee_i a_i,$
then $\tau(a)=\bigvee_i\tau(a_i).$   We recall that if $\tau$ is a
monotone $\sigma$-complete $n$-state-morphism-operator, then it
preserves all existing countable suprema and infima existing in $E,$
and we call it a $\sigma$-{\it complete
$n$-state-morphism-operator.}

Let $f: \mathcal S(E) \to [0,1]$ be any function; we set
$N(f):=\{s\in \partial_e\mathcal S(E): f(s)\ne 0\}.$

Now we present a characterization of $\sigma$-complete
$n$-state-operators on  monotone $\sigma$-complete effect algebras
with (RDP).

\begin{Theorem}\label{th:5.3}
Let $\tau$ be a monotone $\sigma$-complete $n$-state-operator on a
monotone $\sigma$-complete  effect algebra $E$ with {\rm (RDP)}.
Then there is an affine continuous function $g$ defined on $\mathcal
S(E)$ into itself such that $g^n=g,$ $g(s)(E) \subseteq s(E)$ for
any discrete extremal state $s$ on $E$ and $\widehat \tau(\hat
a)=\hat a\circ g,$ $a\in E.$

Conversely, let $g$ be an affine continuous function on $\mathcal
S(E)$ into itself such that $g^n=g,$ $g(s)(E) \subseteq s(E)$ for
any discrete extremal state $s.$ Then the mapping $\tau_g$ defined
on $\widehat E$ by $\tau_g(\hat a):= \hat a\circ g,$ $a\in E,$ is a
monotone $\sigma$-complete $n$-state-operator on $\widehat E.$

In addition, if $\tilde \tau_g$ is defined on $E$ via $\tilde
\tau_g(a)=\tau_g(\hat a),$ $a \in E,$ then $\tilde \tau_g$ is a
monotone $\sigma$-complete $n$-state-operator on $E,$ and
$g(s)=s\circ \tilde \tau_g,$ $s \in \mathcal S(E).$

\end{Theorem}

\begin{proof}
Since $E$ is monotone $\sigma$-complete, due to Theorem
\ref{th:5.2}, $E$ is isomorphic to $A(E)$ as defined by (4.1). By
Proposition \ref{pr:5.1}, there is an affine and continuous function
$g: \mathcal S(E) \to \mathcal S(E)$ such that $g^n = g$ and
$g(s)(E) \subseteq S(E)$ for any discrete extremal state $s$ on $E,$
and the mapping $\tau_g: A(E)\to A(E)$ defined by $\tau_g(f):=
f\circ g,$ $f\in A(E),$ is an $n$-state-operator on $A(E).$

In what follows, we show that $\tau_g$ is monotone
$\sigma$-complete.

Assume that $a=\bigvee_i a_i,$ for $a_i \nearrow a,$ or
equivalently, $\hat a = \bigvee_i \hat a_i.$  Then $\hat a_i \circ g
\le \hat a_{i+1} \circ g \le \hat a\circ g.$

If $a_0(s)= \lim_i \hat a_i(s),$ $s \in \mathcal S(E),$ i.e. $a_0$
is a point limit of continuous functions on a compact Hausdorff
space, due to  Lemma \ref{le:4.1}, the set $N(a_0-\hat a)$ is a
meager set. Similarly, $N(\hat a\circ g - a_0\circ g)$ is a meager
set. If $h= \bigvee_i \hat a_i\circ g,$ then $h \le \hat a\circ g.$
Since $N(h- \hat a\circ g) \subseteq N(h-a_0\circ g)\cup N(a_0\circ
g - \hat a\circ g),$ this yields that $N(h-\hat a \circ g)$ is a
meager set. Due to the Baire Category Theorem that says that no
non-empty open set of a compact Hausdorff space can be a meager set,
we have $N(h-\hat a \circ g)=\emptyset,$ that is $h= \hat a\circ g.$

Finally, let $a\in E$ and $s\in \mathcal S(E).$ Then $(s\circ \tilde
\tau_g)(a) = s(\tilde \tau_g(a))= s(\tau_g(\hat a)) = \hat a(g(s))=
g(s)(a),$ that is $g(s) = s\circ \tilde \tau_g$ for any $s \in
\mathcal S(E).$
\end{proof}

\section{Loomis--Sikorski Theorem for  $n$-State Effect Algebras}

In the present section we will formulate and prove the first main result of the paper.

Let $E$ be a monotone $\sigma$-complete effect algebra with (RDP).
By Theorem \ref{th:5.2}, $\widehat E=A(E)$ but $\widehat E$ is not
necessarily an effect-tribe. Let $\mathcal T(E)$ be the effect-tribe
of functions from $[0,1]^{ \mathcal S(E)}$ generated by $\widehat
E=A(E).$

\begin{Proposition}\label{pr:6.1}
Let $E$ be a monotone $\sigma$-complete effect algebra with {\rm
(RDP)} and let $g$ be an affine continuous function on $\mathcal
S(E)$ into itself such that $g^n= g$ and $g(s)\in s(E)$ for any
discrete $s \in \partial_e\mathcal S(E).$  Then the operator
$\mathcal T_g$ defined on $\mathcal T(E)$ by $\mathcal T_g(f)=f\circ
g,$ $f \in \mathcal T(E),$ is a monotone $\sigma$-complete
$n$-state-operator on $\mathcal T(E)$ and is the unique extension of
the monotone $\sigma$-complete $n$-state-operator $\tau_g$ on $A(E)$
defined by $\tau_g(f)=f\circ g,$ $f \in A(E).$
\end{Proposition}

\begin{proof}
First of all we show that $\mathcal T_g$ is a well-defined operator
on $\mathcal T(E),$ that is, if $f \in \mathcal T(E),$ then $f\circ
g \in \mathcal T(E).$  Let $\mathcal T'$ be the set of all $f\in
\mathcal T(E)$ such that $f\circ g \in \mathcal T(E).$ Then
$\mathcal T'$ contains  $A(E)=\widehat E$ and if $f\in \mathcal T',$
then $1-f\in \mathcal T'.$ Now let $f_1,f_2 \in \mathcal T'$ be such
that $f_1\le f_2',$ then $f_1 + f_2$  belongs to $\mathcal T'.$
Hence, if $\{f_i\}$ is a sequence of monotone functions from
$\mathcal T',$ then, for $f=\lim_i f_i,$ we have $f\circ g = \lim_i
f_i\circ g\in \mathcal T'$. This implies that $\mathcal T'$ is an
effect-tribe generated by $A(E)$, consequently, $\mathcal
T'=\mathcal T(E)$ and $\mathcal T_g$ is a monotone $\sigma$-complete
$n$-state-operator on $\mathcal T(E)$ that is an extension of
$\tau_g.$

Now if $\tau$ is any monotone  $\sigma$-complete $n$-state-operator
on $\mathcal T(E)$ that is an extension of $\tau_g,$ then again the
set of elements $f \in \mathcal T(E)$ such that $\tau(f)=\mathcal
T_g(f)$ is a tribe containing $A(E).$ Thus, it has to be $\mathcal
T(E)$ and $\tau=\mathcal T_g.$
\end{proof}

Let $(E_1,\tau_1)$ and $(E_2,\tau_2)$ be $n$-state effect algebras.
A homomorphism $h:E_1 \to E_2$ is said to be a {\it
state-homomorphism} if $h\circ \tau_1 = \tau_2 \circ h.$ Similarly,
we define  a {\it monotone} $\sigma$-{\it complete
state-homomorphism} if $(E_1,\tau_1)$ and $(E_2,\tau_2)$ are
monotone $\sigma$-complete state effect algebras and $h$ is a
state-homomorphism such that if $a_i \nearrow a,$ then $h(a_i)
\nearrow h(a).$

We now generalize the Loomis--Sikorski Theorem for monotone
$\sigma$-complete $n$-state effect algebras.

\begin{Theorem}[Loomis--Sikorski Theorem]\label{th:6.2}
Let $(E,\tau)$ be a monotone $\sigma$-complete $n$-state effect
algebra with {\rm (RDP)}. Then there are a monotone
$\sigma$-complete $n$-state effect algebra $(\mathcal T,\mathcal
T_g),$ where $\mathcal T$ is an effect-tribe of functions from
$[0,1]^\Omega$ satisfying {\rm (RDP)}, a function $g:\Omega \to
\Omega$ such that $g^n = g$ and $f\circ g\in \mathcal T$ for any $f
\in \mathcal T,$ such that $\mathcal T_g(f):=f\circ g,$ $f\in
\mathcal T,$ is a monotone $\sigma$-complete $n$-state-operator on
$\mathcal T.$ Moreover, there is a monotone $\sigma$-complete
state-homomorphism $h$ from $\mathcal T$ onto $E$ such that $h\circ
\mathcal T_g= \tau\circ h.$
\end{Theorem}

\begin{proof}
Let $E$ be a monotone  $\sigma$-complete effect algebra with a
monotone  $\sigma$-complete $n$-state-operator $\tau.$  We
isomorphically embed $E$ onto $\widehat E.$ We set $\Omega =\mathcal
S(E),$ then $\Omega$ is a compact Hausdorff topological space and
$\widehat E = A(E)$ by Theorem \ref{th:5.2}.  Let $\mathcal T(E)$ be
the effect-tribe of functions from $[0,1]^\Omega$ that is generated
by $\widehat E.$ According to Theorem \ref{th:5.3}, the function
$g:\mathcal S(E) \to \mathcal S(E)$ defined by $g(s)=s\circ \tau,$
$s \in \mathcal S(E),$ is continuous and $g^n = g.$  The mapping
$\mathcal T_g: \mathcal T(E)\to \mathcal T(E)$ defined by $\mathcal
T_g(f)=f\circ g,$ $f \in \mathcal T(E),$ is by Proposition
\ref{pr:6.1} a monotone $\sigma$-complete $n$-state-operator on
$\mathcal T(E)$ and by the same Proposition, it is a unique
extension of the monotone $\sigma$-complete $n$-state-operator
$\tau_g$ on $\mathcal E$ defined by $\tau_g(\hat a)= \hat a\circ g,$
$a\in E.$

Let us denote by ${\mathcal T}$ the class of all functions $f \in
[0,1]^{\mathcal S(E)}$ such that there is an element $a \in E$ with
a meager set $N(f-\hat a):=\{s\in
\partial_e \mathcal S(E): f(s)\ne \hat a(s)\}$;
in which case we write  $f\sim a.$

If $a_1$ and $a_2$ are two elements of $E$ such that $f\sim  a_1$
and $f\sim  a_2, $ then $N(\hat a_1-\hat a_2)\subseteq N(f-\hat
a_1)\cup N(f-\hat a_2)$ proving that $N(\hat a_1-\hat a_2)$ is a
meager set. $\hat a_1$ and $\hat a_2$ are continuous functions, and
due to the Baire Category Theorem, we have $\hat a_1=\hat a_2.$

Therefore, the mapping $h:\mathcal T\to E$ defined by $h(f)=a$ if
$f\sim a$ is a well-defined mapping.

In what follows, we show that $\mathcal T$ is an effect-clan with
(RDP) such that $\mathcal T=\mathcal T(E)$ and $h$ is a monotone
$\sigma$-complete $n$-state-homomorphism on $A(E).$

Let $f_1$ and $f_2$ be  two functions from ${\mathcal T}$ with $f_1
\le f_2$. Choose $b_1, b_2 \in E$ such that $f_i\sim b_i$ for
$i=1,2$. We assert that $b_1 \le b_2$.

Indeed, we have $\{s \in \partial_e {\mathcal S}(E):\ 0< \hat b_1(s)
- \hat b_2(s) \} \subseteq N(f_1 - \hat b_1) \cup N(f_2 - \hat
b_2)$.

The Baire Category  Theorem applied to $\partial_e {\mathcal S}(E)$
implies that no nonempty open set of $\partial_e {\mathcal S}(E)$
can be a meager set, whence $s(b_1) \le s(b_2)$ for any $s \in
\partial_e {\mathcal S}(E)$, and consequently $\hat s(b_2 - b_1)\ge 0$
for any $s \in {\mathcal S}(E).$  Because our $\Aff(\partial_e \mathcal
S(E),1)$   is Archimedean,   this yields that the set $\partial_e
\mathcal S(E)$ is order determining that entails $b_1 \le b_2.$

It is clear that the set ${\mathcal T}$ is closed under the
formation of complements $f \mapsto 1-f$, and   it  contains $\{\hat
a:\ a \in E\}.$

If $f,g \in {\mathcal T}$, $f\le 1- g$, and $N(f-\hat a), N(g-\hat
b)$ are meager subsets of $\partial_e {\mathcal S}(E)$, then  $a \le
b'$, so that $a+b \in E$. Hence, $N(f+g - \widehat{(a+b)})$ is
meager, i.e., $f+g \in {\mathcal T}$, and ${\mathcal T}$ is an
effect-clan.

To show that ${\mathcal T}$ is an effect-tribe it is necessary to
verify that ${\mathcal T}$ is closed under limits of non-decreasing
sequences from ${\mathcal T}.$  Let $\{f_i\}$ be a non-decreasing
sequence of elements from ${\mathcal T}$. For any $f_i$, choose by
(i) a unique $b_i \in E$ such that $N(f_i - \hat b_i)$ is a meager
subset of $\partial_e {\mathcal S}(E)$.

Denote by $f = \lim_i f_i,$ $b = \bigvee_{i=1}^\infty b_i$, $ b_0 =
\lim_i \hat b_i.$ Then  $b \in E$ and $\hat b \in {\mathcal T}.$ We
have
$$ N(f - \hat b) \subseteq N(f - b_0) \cup N(\hat b -  b_0)
$$
and $N(f - b_0) = \{s\in \partial_e {\mathcal S}(E):\ f(s) <
b_0(s)\} \cup \{s\in \partial_e {\mathcal S}(E):\ b_0(s) < f(s)\}$.

If $s \in  \{s\in \partial_e {\mathcal S}(E):\ f(s) < b_0(s)\},$
then there is an integer $i \ge 1$ such that $f(s) < \hat b_i(s) \le
b_0(s).$ Hence $f_i(s) \le f(s) < \hat b_i(s) \le b_0(s)$ so that $s
\in \{s\in \partial_e {\mathcal S}(E):\ f_i(s) < \hat b_i(s)\}$.

Similarly we can prove that if $s\in \{s\in \partial_e {\mathcal
S}(E):\ b_0(s) < f(s)\}$, then there is an integer $i \ge 1$ such
that $s\in \{s\in \partial_e {\mathcal S}(E):\ \hat b_i(s) <
f_i(s)\}.$

These  two cases yield
$$ N(f - b_0) \subseteq \bigcup_{i=1}^\infty N(\hat b_i - f_i)
$$
which is a meager subset of $\partial_e {\mathcal S}(E)$.

Since $b = \bigvee_{i=1}^\infty b_i,$ we conclude that $\bigwedge_i
b_i' = b'$ and $\bigwedge_i (b_i' - b') = 0$. Due to Lemma
\ref{le:4.1}, we have that $N(\hat b - b_0)$ is a meager subset of
$\partial_e {\mathcal S}(E)$.  Hence, $f \in {\mathcal T}.$

Consequently, we have proved that ${\mathcal T}$ is an effect-tribe.
Since $\mathcal T$ contains $A(E),$ we have $\mathcal T=\mathcal T(E).$

Now we concentrate to show  that $\mathcal T$ satisfies (RDP). Let
$f \le g+h$, i.e., $f(s) \le g(s) +h(s)$ for any $s\in {\mathcal
S}(E)$. By (ii) and (i) there are unique  elements $a,b,c \in E$
such that $f \sim a,$ $g \sim b$ and $h\sim c$, and $a \le b+c$.
Since $E$ satisfies (RDP), there are two elements $b_1,c_1 \in E$
such that $a = b_1 +c_1$ and $b_1 \le b$ and $c_1 \le c$.
Consequently, there is a meager set $K$ of $\partial_e {\mathcal
S}(E)$ such that $f(s) =\hat a(s)$, $g(s) = \hat b(s)$ and $h(s) =
\hat c(s)$ for any $s \in {\mathcal S}(E) \setminus K.$ For the
functions $f|_K,$ $g|_K$ and $h|_K$ defined on $K$ (i.e. the system
of all $[0,1]$-valued functions on $K$), (RDP) trivially holds,
i.e., there are two functions $g_0$ and $h_0$ defined on $K$ such
that $f(s) = g_0(s) +h_0(s),$ $g_0(s) \le g(s)$ and $h_0(s) \le
h(s)$ for any $s \in K$. Let us define functions $g_1$ and $h_1$ on
${\mathcal S}(E)$ by $g_1(s) = \hat b_1(s)$, $h_1(s) = \hat c_1(s)$
for any $s \in {\mathcal S}(E)\setminus K$ and $ g_1(s) = g_0(s)$,
$h_1(s) = h_0(s)$ for any $s \in K$. Then $f= g_1 +h_1$, $g_1 \le g$
and $h_1 \le h,$ and $g_1 \sim a_1$ and $h_1 \sim b_1$, which proves
that ${\mathcal T}$ has (RDP).

Due to the definition of $\mathcal T$ and the previous steps, the
mapping $h:\, {\mathcal T} \to E$ defined by $h(f) = b$ iff $f \sim
b$ is a surjective and monotone $\sigma$-complete homomorphism from
$\mathcal T$ onto $E.$

Finally,  let  $f\in \mathcal T$ and $a\in E$ be such $h(f)=a.$ Then
$f\sim a$ so that $N(f-\hat a)$ is a meager set. Then $N(f\circ g -
\hat a\circ g)= g^{-1}(N(f-\hat a))$ is also meager.  By Theorem
\ref{th:5.3}, we have $h(\mathcal T_g (f))=\tau (a)= \tau(h(f)).$
\end{proof}

\section{Stone Dualities and F-spaces}

We present the second main result of the paper, Stone Dualities
between some categories of effect algebras and F-spaces, Theorem
\ref{th:7.10}.

We say that a topological space $\Omega$ is an {\it F-space} if any
two disjoint open $F_\sigma$ subsets of $\Omega$ have disjoint
closures.  For example, every basically disconnected compact Hausdorff
space is an F-space.

We say that a poset $E$ satisfies the {\it countable interpolation
property} provided that for any two sequences $\{x_i\}$ and
$\{y_j\}$ of elements of $E$ such that $x_i \le y_j$ for all $i,j,$
there exists an element $z \in E$ such that $x_i \le z \le y_j$ for
all $i,j.$

\begin{Theorem}\label{th:7.1}  Let $\Omega$ be a Bauer simplex. The
following statements are equivalent.

\begin{enumerate}

\item[{\rm (i)}] $\Gamma(\Aff(\Omega),1)$ is a weakly divisible
effect algebra with countable interpolation.

\item[{\rm (ii)}] $\Gamma(\mbox{\rm C}(\partial_e\Omega),1)$ is a weakly divisible
effect algebra with countable interpolation

\item[{\rm (iii)}] $\partial_e\Omega$ is an F-space.

\end{enumerate}
\end{Theorem}

\begin{proof}
It is clear that both $\Gamma(\Aff(\Omega),1)$ and
$\Gamma(\mbox{C}(\Omega),1)$ are weakly divisible as well as divisible
effect algebras.

(i) $\Rightarrow$ (ii). Assume that $\{f_n\}$ and $\{g_m\}$ are two
sequences of continuous functions from
$\Gamma(\mbox{C}(\partial_e\Omega),1)$ such that $f_n \le g_m$.  Let
$\tilde f_n$ and $\tilde g_m$ be unique extensions to affine
continuous functions on $\Omega$ of $f_n$ and $g_m,$ respectively.
Then $\tilde f_n \le \tilde g_m$ for all $n,m.$  The countable
interpolation on $\Gamma(\Aff(\Omega),1)$ yields that there is an
affine function $h \in \Gamma(\Aff(\Omega),1)$ such that $\tilde f_n
\le h \le \tilde g_m$ for all $n,m.$  If $h_0$ is the restriction of
$h$ onto $\partial_e \Omega,$ then $f_n \le h_0 \le g_m$ for all
$n,m$ so that $\Gamma(\mbox{C}(\partial_e\Omega),1)$ satisfies countable
interpolation.

(ii) $\Rightarrow$ (i).  Since $\Omega$ is a Bauer  simplex, and
consequently a Choquet one, $\Aff(\Omega)$ is an interpolation
group, \cite[Thm 11.4]{Goo}. Let $\{f_n\}$ and $\{g_m\}$ be two
sequences of continuous affine functions from
$\Gamma(\Aff(\Omega),1)$ such that $f_n \le g_m$ for each $n,m.$
Since the functions are continuous, there is a continuous function
$h \in \mbox{C}(\partial_e\Omega)$ such that $f_n(x)\le h(x)\le g_m(x)$ for
all $n,m$ and $x\in \partial_e \Omega.$ By the Tietze Theorem,
\cite[Prop II.3.13]{Alf}, $h$ can be uniquely extended to an affine
function $\tilde h.$  Since $f_n(x)\le \tilde h(x) \le g_m(x)$ for
all $x\in
\partial_e \Omega,$ by \cite[Cor 5.20]{Goo}, this implies $f_n(x)\le
\tilde h(x)\le g_m(x)$ for any $x \in \Omega.$

(ii) $\Leftrightarrow$ (iii)  According to \cite[Thm 1.1]{See}, a
compact Hausdorff topological space $K$ is an F-space iff $\mbox{C}(K)$
satisfies countable interpolation. By \cite[Prop 16.3]{Goo},
$\Gamma(\mbox{C}(\partial_e\Omega),1)$ satisfies countable interpolation
iff $(\mbox{C}(\partial_e\Omega),1)$ satisfies countable interpolation.
\end{proof}

\begin{Remark}\label{re:7.2}
Here it is necessary to point out that not every MV-algebra
satisfying countable interpolation is $\sigma$-complete.  Due to the
Nakano Theorem \cite[Cor 9.3]{Goo}, if $\Omega$ is a compact
Hausdorff space, then the MV-algebra $\Gamma(\mbox{C}(\Omega),1)$ is
$\sigma$-complete iff $\Omega$ is basically disconnected. Every such
a basically disconnected space $\Omega$ can be expressed as a union
of two nonempty clopen subsets, so that $\Omega$ is not connected.
But due to \cite{GiHe} or \cite[p. 280]{Goo}, there exists an
F-space, $\Omega_0,$ that is connected, so that it is not basically
disconnected. By \cite{See}, $\Gamma(\mbox{C}(\Omega_0),1)$ is an
MV-algebra that satisfies countable interpolation, but due to the
Nakano Theorem it is not $\sigma$-complete.
\end{Remark}

If  $K$ is a compact Hausdorff topological space, let ${\mathcal
B}(K)$ be the Borel $\sigma$-algebra of $K$ generated by all open
subsets of $K.$  Let  ${\mathcal M}_1^+(K)$ denote the set of  all
probability measures, that is, all positive regular
$\sigma$-additive Borel measures $\mu$ on $\mathcal B(K).$  We recall
that a Borel measure $\mu$ is called regular if

$$\inf\{\mu(O):\ Y \subseteq O,\ O\ \mbox{open}\}=\mu(Y)
=\sup\{\mu(C):\ C \subseteq Y,\ C\ \mbox{closed}\}
$$
for any $Y \in {\mathcal B}(K).$

Let $x \in K$ and let $\delta_x$ be the Dirac measure concentrated
at the point $x \in K,$ i.e., $\delta_x(Y)= 1$ iff $x \in Y$,
otherwise $\delta_x(Y)=0;$ then every Dirac measure is a regular
Borel probability measure.   Moreover, \cite[Prop 5.24]{Goo},  the
mapping
$$
\epsilon: x\mapsto \delta_x \eqno(7.1)
$$
gives a homeomorphism of $K$ onto $\partial_e {\mathcal M}_1^+(K).$

Hence, if $K$ is an F-space that is connected, see Remark
\ref{re:7.2}, then $\Omega:= {\mathcal M}_1^+(K)$ gives a
Bauer simplex whose boundary $\partial_e \Omega$ is a connected
compact Hausdorff F-space.  Moreover, $\Gamma(\Aff(\Omega),1)$ gives
by Theorem \ref{th:7.1} a divisible lattice ordered effect algebra
satisfying monotone interpolation and (RDP) that has an order
determining system of states but $\Gamma(\Aff(\Omega),1)$ is not
monotone $\sigma$-complete, see \cite[Thm 4.2]{DDL3}.

\begin{Theorem}\label{th:7.4}  Let $E$ be an effect algebra with
{\rm  (RDP)} and with countable interpolation such that $E$ has an
order determining system of states. Then $E$ is isomorphic to
$A(E),$ where $A(E)$ is defined by {\rm (4.1)}, $E$ is lattice
ordered and $\partial_e \mathcal S(E)$ is an F-space.
\end{Theorem}

\begin{proof}
If $E$ has an order determining system of states, $\mathcal S,$ then
$\mathcal S(E)$ is also order determining. Moreover, $E=\Gamma(G,u)$
for some interpolation unital po-group $(G,u).$ In view of Remark
\ref{re:2.3}, $\mathcal S(G,u)$ is also order determining, so that
$G$ is Archimedean. Applying \cite[Thm 16.19(b)]{Goo}, we have that
$G$ is an $\ell$-group so that  $E$ is a lattice. Due to \cite[Thm
16.14]{Goo}, $E$ is isomorphic with $A(E)$. Hence, $\mathcal S(E)$
is a Bauer simplex and by \cite[Thm 16.22]{Goo}, $\partial_e
\mathcal S(E)$ is an F-space.
\end{proof}

\begin{Remark}\label{re:7.5}  Under the assumptions of Theorem
\ref{th:7.4}, we see that $E$ is in fact a semisimple MV-algebra
(equivalently this means that $\mathcal S(E)$ is order determining)
with countable interpolation whose boundary $\partial_e \mathcal
S(E)$ is an F-space, and vice-versa. Every semisimple MV-algebra
satisfying countable interpolation satisfies the condition of
Theorem \ref{th:7.4}.
\end{Remark}

Let $n\ge 1$ be a fixed integer. Let $\mathcal {DSMEA}_n$ be the
category of (weakly) divisible state-morphism effect algebras whose
objects are couples $(E,\tau),$ where $E$ is an effect algebra
satisfying (RDP) and countable interpolation with an order
determining system of states, and $\tau$ is an
$n$-state-morphism-operator on $E$; and a morphism from
$(E_1,\tau_1)$ to $(E_2,\tau_2)$ is any homomorphism $h: E_1 \to
E_2$ that preserves all existing meets and joins in $E_1$ such that
$h\circ \tau_1=\tau_2\circ h.$ We note that $\mathcal {DSMEA}_n$ is
a category.

Let $\mathcal {BSF}_n$ be the category of Bauer simplices whose
objects are pairs $(\Omega,g),$ where $\Omega\ne \emptyset$ is a
Bauer simplex  such that $\partial_e \Omega$ is an F-space, and $g:
\Omega \to \Omega$ is an affine continuous function such that $g^n =
g$, $g:\partial_e \Omega \to
\partial_e \Omega.$ Morphisms from $(\Omega_1,g_1)$ into
$(\Omega_2,g_2)$ are continuous affine  functions $p: \Omega_1 \to
\Omega_2$ such that $p:\partial_e \Omega_1 \to
\partial_e \Omega_2$ and $p \circ g_1= g_2\circ p.$  Then $\mathcal{BSF}_n$ is also
a category.

Now we reformulate the substantial part of Proposition \ref{pr:5.1}
for state-morphism-operators on lattice ordered  effect algebras
with (RDP) and with an ordering system of states.

\begin{Proposition}\label{pr:7.6}  Let $\tau$ be an $n$-state-morphism
on a lattice ordered effect algebra satisfying {\rm (RDP)} and
countable interpolation and with an ordering system of states. Then
$\tau$ satisfies {\rm (ESP)} and there is an affine continuous
function $g$ from $\mathcal S(E)$ into itself such that it maps
$\partial_e \mathcal S(E)$ into itself, $g^n = g,$ $g(s)(E)
\subseteq S(E)$ for any discrete state $s$ on $E.$ Moreover, the
mapping $\tau_g: A(E)\to A(E)$ defined by $\tau_g(f)=f\circ g,$ $f
\in A(E),$ where $A(E)$ is defined by $(4.1),$ is an
$n$-state-morphism-operator on $A(E).$ In addition, $(E,\tau)$ and
$(A(E),\tau_g)$ are isomorphic $n$-state-morphism effect algebras.
\end{Proposition}

\begin{proof}  The mapping $\psi: a\mapsto \hat a,$ defined by $\hat
a(s):= s(a),$ $a \in E,$ $(s \in \mathcal S(E))$ is by Theorem
\ref{th:7.4} an isomorphism from $E$ into the effect-clan $A(E)$
defined by (4.1). Since $E$ is in fact an MV-algebra, $\tau$
satisfies (ESP) property, and by Proposition \ref{pr:5.1}, the
mapping $g:\mathcal S(E) \to \mathcal S(E),$ defined by $g(s):=
s\circ \tau,$ $s \in \mathcal S(E),$ is affine and continuous, $g^n
= g,$ and it maps any extremal state on $E$ into an extremal state.
Moreover, $g(s)(E)\subseteq s(E)$ for any discrete extremal state
$s.$

The mapping $\psi$ is an isomorphism of effect algebras.  We show that it also
preserves all existing meets and joins in $E.$ We know already from
Theorem \ref{th:7.4} and Remark \ref{re:7.5} that $E$ is in fact an
MV-algebra. Then for any discrete state $s$ on $E$ we have
$\psi(a\wedge b)(s) = s(a\wedge b) = \min\{s(a),s(b)\}=
\min\{\psi(a)(s),\psi(b)(s)\}$ due to basic properties of extremal
states on MV-algebras. Hence $(\psi(a)\wedge \psi(b))(s) =
\psi(a\wedge b)(s)$ for any discrete state $s.$  Since
$\psi(a)\wedge \psi(b)\in A(E),$ we have that  $(\psi(a)\wedge
\psi(b))(s) = \psi(a\wedge b)(s)$ for any  state $s$ on $E.$

Therefore, the mapping $\tau_g$ is a well-defined state-operator on
$A(E).$  For all $f_1,f_2 \in A(E)$ we have $(f_1\wedge f_2)(s) =
\min\{f_1(s),f_2(s)\}$ for any $s\in \mathcal S(E).$ Hence,
$(\tau_g(f_1\wedge f_2))(s)= (f_1\wedge f_2)(g(s)) =
\min\{f_1(g(s)),f_2(g(s))\} = (\tau_g(f_1)\wedge \tau_g(f_2))(s).$
So that $\tau_g$ is an $n$-state-morphism-operator on $A(E)$.

Now it is easy to verify that $\psi\circ \tau= \tau_g \circ \psi$
proving that $(E,\tau)$ and $(A(E),\tau_g)$ are isomorphic
$n$-state-morphism effect algebras because $\psi$ preserves all
existing meets and joins in $E.$
\end{proof}

Define a morphism $S: \mathcal {DSMEA}_n \to \mathcal{BSF}_n$ by
$S(E,\tau) = (\mathcal S(E),g),$ where $g$ is an affine continuous
function from $\mathcal S(E) \to \mathcal S(E)$ such that
$g(s)=s\circ \tau,$ $s \in \mathcal S(E),$ $g^n = g$ that is
guaranteed by Proposition \ref{pr:7.6}.

\begin{Proposition}\label{pr:7.7}
The function $S: \mathcal{DSMEA}_n\to \mathcal{BSF}_n$ defined by
$S(E,\tau) = ({\mathcal S}(E), g)$ is a contravariant functor from
$\mathcal{DSMEA}_n$ into $\mathcal{BSF}_n.$
\end{Proposition}

\begin{proof}  Let $(E,\tau)$ be an object from $\mathcal{DSMEA}_n.$ By
Theorem \ref{th:7.4}, $(E,\tau)$ is isomorphic with $(A(E),\tau_g),$
where $A(E)$ is defined by (4.1) and $g$ is an affine continuous
function on $\Omega:=\mathcal{S}(E)$ into itself such that it maps
$\partial_ e\mathcal S(E)$ into itself, $g^n= g$ and $g(s):=s\circ
\tau$ for any $s \in \mathcal{S}(A).$

Let $h$ be any morphism from $(E,\tau)$ into $(E',\tau').$  Define a
mapping  $S(h): \mathcal{S}(E') \to \mathcal{S}(E)$ by $S(h)(s'):=
s'\circ h,$ $s' \in \mathcal{S}(E').$  Then $S(h)$ is affine,
continuous and $g\circ S(h) = S(h)\circ g'.$ Indeed, let $s' \in
{\mathcal S}(E').$ Then $S(h) \circ g'\circ s' = (g' \circ s')\circ
h= g'\circ (s'\circ h)= s'\circ h \circ \tau = s' \circ \tau' \circ
h = S(h)\circ s'\circ \tau' = S(h) \circ g'.$
\end{proof}

Given an convex compact Hausdorff topological space $\Omega \ne
\emptyset,$ let
$$ E(\Omega) := \Gamma(\Aff(\Omega),1). \eqno(7.2)
$$
Then $E(\Omega)$ is a weakly divisible effect algebra with a
determining system of states.

Define a morphism $T: \mathcal{BSF}_n\to \mathcal{DSMEA}_n$ via
$T(\Omega,g)=(E(\Omega),\tau_g)$, where
$E(\Omega)=\Gamma(\Aff(\Omega),1),$ $\tau_g(f):=f\circ g,$ $f\in
E(\Omega),$ and if $p: (\Omega,g)\to (\Omega',g'),$ then $T(p)(f):
E(\Omega') \to E(\Omega)$ is defined by $T(p)(f) := f\circ p,$ $f
\in E(\Omega').$

\begin{Proposition}\label{pr:7.8} The function $T: \mathcal{BSF}_n\to
\mathcal{DSMEA}_n$ is a contravariant functor from $\mathcal{BSF}_n$
to $\mathcal{DSMEA}_n.$
\end{Proposition}

\begin{proof}  If $(\Omega,g)$ is an object from $\mathcal{DSMEA}_n,$
then $E(\Omega)$ is a (weakly) divisible effect algebra satisfying
(RDP) and with an ordering system of states. In addition, by Theorem
\ref{th:7.1}, $E(\Omega)$ satisfies countable interpolation. The
mapping $\tau_g(f):= f\circ g,$ $f\in E(\Omega),$ is by Proposition
\ref{pr:7.6} a  state-morphism-operator on $E(\Omega).$ Therefore,
$T(\Omega,g)=(E(\Omega),\tau_g)\in \mathcal{DSMEA}_n.$

Now let  $p: (\Omega,g) \to (\Omega',g')$ be a morphism, i.e. an
affine continuous function $p:\Omega \to \Omega'$ such that
$p:\partial_e \Omega \to \partial_e \Omega'$ and $p \circ g= g'\circ
p.$ We assert $\tau_{g} \circ T(p) = T(p)\circ \tau_{g'}.$ Check:
for any $f \in E(\Omega'),$ we have $\tau_{g} \circ T(p)\circ f=
\tau_{g} \circ (T(p) \circ f) = \tau_{g} \circ (f\circ p) = (f \circ
p) \circ g = f\circ (p \circ g) = f \circ (g' \circ p)= (f\circ
g')\circ p= T(p) \circ (f\circ g')= T(p) \circ (\tau_{g'} \circ f) =
T(p) \circ \tau_{g'} \circ f .$
\end{proof}

\begin{Remark}\label{re:7.9}
It is worthy to remark that due to \cite[Thm 7.1]{Goo}, if $\Omega$
is a compact convex subset of a locally convex Hausdorff space, then
the evaluation mapping $p:\ \Omega \to \mathcal S(E(\Omega)) $
defined by $p(x)(f)=f(x)$ for all $f \in E(\Omega)$ $(x \in \Omega)$
is an affine homeomorphism of $\Omega$ onto $\mathcal S(E(\Omega)).$
\end{Remark}

\begin{Theorem}[Stone Duality Theorem]\label{th:7.10}
The categories $\mathcal{BSF}_n$ and $\mathcal{DSMEA}_n$ are dual.
\end{Theorem}

\begin {proof}
We show that the conditions of \cite[Thm IV.1]{Mac} are fulfilled,
i.e. $T\circ S(E,\tau) \cong (E,\tau)$ and $S\circ T(\Omega,g) \cong
(\Omega,g)$ for all $(E,\tau) \in \mathcal{DSMEA}_n$ and $(\Omega,g)
\in \mathcal{BSF}_n.$

(i) Propositions \ref{pr:7.7}--\ref{pr:7.8} entail that if $(E,\tau)
\in \mathcal{DSMEA}_n,$ then $T\circ S(E,\tau) = T(\mathcal{S}(E),g)
= (E(\mathcal{S}(E)),\tau_g) \cong (E,\tau).$

(ii)  Now let $(\Omega,g)$ be any object from $\mathcal{BSF}_n.$  By
Remark \ref{re:7.9}, $\Omega$ and $\mathcal S(E(\Omega))$ are
affinely homeomorphic under the evaluation mapping $p:\Omega \to
\mathcal S(E(\Omega)).$  We assert $p\circ g=g'\circ p.$

Let  $x \in \Omega$ and $f \in E(\Omega)$ be arbitrary. Then
$s=p(x)$ is a state from $\mathcal S(E(\Omega)).$  The function $g':
\mathcal S(E(\Omega)) \to \mathcal S(E(\Omega))$ is defined by the
property $g'(s)=s\circ \tau_g.$  Since $g'(s)= g'(p(x)),$ we get
\begin{eqnarray*} (g'\circ p)(x)(f)&=& g'(p(x))(f) = (g'(s))(f)= (s\circ \tau_g)(f)\\
 &=&p(x) \circ (\tau_g(f))
= p(x) \circ (f\circ g) =f(g(x)).
\end{eqnarray*}
On the other hand,
$$ (p\circ g)(x)(f) = p(g(x))(f)= f(g(x))$$
that proves $p\circ g=g'\circ p.$  Hence, the categories
$\mathcal{BSF}_n$ and $\mathcal{DSMEA}_n$ are dual.
\end{proof}


\end{document}